\documentclass{amsproc}

\usepackage{amsmath}
\usepackage{amssymb}
\usepackage{rotating}
\usepackage{graphicx}
\usepackage[tableposition=below]{caption}
\usepackage{bm}

\DeclareGraphicsExtensions{.png,.pdf,.eps}
\usepackage[all,2cell,cmtip]{xy}
\usepackage{tikz}
\usepackage{color}
\usepackage{longtable}
\usepackage{soul}
\newtheorem{theorem}{Theorem}[section]
\newtheorem{lemma}[theorem]{Lemma}

\newtheorem{proposition}[theorem]{Proposition}

\theoremstyle{definition}

\newtheorem{remark}[theorem]{Remark}
\newtheorem{assumption}[theorem]{Assumption}

\newtheorem{question}[theorem]{Question}

\def\P{{\mathbb P}}

\def\cE{{\mathcal E}}

\def\cO{{\mathcal{O}}}

\def\cOperatorname#1{\mathop{\rm #1}\nolimits}

\def\det{\cOperatorname{det}}

\def\ME{{\cOperatorname{ME}}}

\newcommand{\cME}[1]{\cOverline{\ME}}

\makeindex

\begin{document}

\title{Positivity of the exterior power of the tangent bundles}

\author{Kiwamu Watanabe}
\date{\today}
\address{Department of Mathematics, Faculty of Science and Engineering, Chuo University.
1-13-27 Kasuga, Bunkyo-ku, Tokyo 112-8551, Japan}
\email{watanabe@math.chuo-u.ac.jp}
\thanks{The author is partially supported by JSPS KAKENHI Grant Number 21K03170 and the Sumitomo Foundation Grant Number 190170.}

\subjclass[2010]{14J40, 14J45, 14M22.}
\keywords{}

\begin{abstract} Let $X$ be a complex smooth projective variety such that the exterior power of the tangent bundle $\bigwedge^{r} T_X$ is nef for some $1\leq r<\dim X$. We prove that, up to an \'etale cover, $X$ is a Fano fiber space over an Abelian variety. This gives generalizations of the structure theorem of varieties with nef tangent bundle by Demailly, Peternell and Schneider \cite{DPS94} and that of varieties with nef $\bigwedge^{2} T_X$ by the author \cite{Wat21}. Our result also gives an answer to a question raised by Li, Ou and Yang \cite{LOY19} for varieties with strictly nef $\bigwedge^{r} T_X$ when $r < \dim X$. 
\end{abstract}

\maketitle

\section{Introduction} Positivity for vector bundles such as ampleness and nefness has left its mark on the study of algebraic geometry. Let $X$ be a complex smooth projective variety of dimension $n$; we focus on the positivity of the tangent bundle $T_X$, which reflects the global geometry of $X$.  As a generalization of the Hartshorne-Frankel conjecture solved by Mori \cite{Mori79} (see also \cite{SiuYau80} by Siu and Yau), Campana and Peternell \cite{CP91} studied the structure of smooth projective varieties with nef tangent bundle, paying special attention to $3$-folds. In higher dimensional case, Demailly, Peternell and Schneider obtained the following structure theorem:

\begin{theorem}[{\cite[Main Theorem]{DPS94}}]\label{them:DPS} If $T_X$ is nef, then there exists a finite \'etale cover $X'\to X$ such that $X'$ is a locally trivial fibration $\varphi: X' \to {\rm Alb}(X')$ whose fibers are a Fano variety. 
\end{theorem}

Some years ago, Cao and H\"oring extended Theorem~\ref{them:DPS} to a more general setting:

\begin{theorem}[{\cite[Theorem~1.3]{CH19}}]\label{them:CH19}  If the anticanonical divisor $-K_X$ is nef, then there exists a finite \'etale cover $X'\to X$ such that $X'\cong Y\times Z$ where $K_Y$ is trivial and $Z$ is a  locally trivial fibration $\varphi: Z \to {\rm Alb}(Z)$ with a rationally connected fiber. 
\end{theorem}

In general a fiber of $\varphi$ in Theorem~\ref{them:CH19} is not a Fano variety, because there exists a lot of rationally connected projective varieties with nef anticanoncial divisor which is not Fano (for instance, consider weak Fano varieties). The main result of this paper is a generalization of Theorem~\ref{them:DPS}:

\begin{theorem}\label{them:str} Let $X$ be a smooth projective variety of dimension $n$. Assume that $\bigwedge^{r} T_X$ is nef for some $1\leq r<n$. Then if we take a suitable finite \'etale cover $\tilde{X} \to X$, there exists a locally trivial fibration $\varphi: \tilde{X} \to A$ such that the fiber $F$ is a Fano variety and $A$ is an Abelian variety.
Moreover, if $\dim A\geq r-1$, then $T_{X}$ is nef; otherwise $\bigwedge^{r-\dim A}T_{\tilde{X}/A}$ is nef. 
\end{theorem}

This theorem reduces the study of smooth projective varieties with nef $\bigwedge^r T_X$ ($r<\dim X$) to that of Fano varieties. 
For $r=1$, Theorem~\ref{them:str} is nothing but Theorem~\ref{them:DPS}; for $r=2$ this was obtained in \cite[Theorem~1.5]{Wat21}. The proof of \cite[Theorem~1.5]{Wat21} involves the deformation theory of rational curves and some complicated arguments. On the other hand, in this short paper, we give a really simple proof of Theorem~\ref{them:str}. Our proof relies on two key ingredients; one is Theorem~\ref{them:DPS}; the other is a recent result of Gachet \cite{Gac22}. In \cite[Theorem~1.2]{Gac22}, she proved that for a smooth projective variety $X$ of dimension $n$ if $\bigwedge^{n-1} T_X$ is strictly nef, then $X$ is a Fano variety. Her proof works if we replace the assumption that $\bigwedge^{n-1} T_X$ is strictly nef by the assumption that $X$ is rationally connected and $\bigwedge^{n-1} T_X$ is nef. Moreover by using a result by Laytimi and Nahm \cite{LayNah05}, we see that if $\bigwedge^{r} T_X$ is nef for some $r<n$, then so is $\bigwedge^{n-1} T_X$. Thus we have the following:
\begin{proposition}\label{prop:RC:Fano} 
Let $X$ be a smooth projective variety of dimension $n$. Assume that $X$ is rationally connected and $\bigwedge^{r} T_X$ is nef for some $1\leq r<n$. Then $X$ is a Fano variety. 
\end{proposition} 
Remark that, combining with \cite[Theorem~1.2]{LOY19}, Proposition~\ref{prop:RC:Fano} gives an affirmative answer to the following question by Li, Ou and Yang when $r < \dim X$: 
\begin{question}[{\cite[Remark~5.3]{LOY19}, \cite[Conjecture~4.9]{LOY21}, \cite[Question in Section 1]{Gac22}}]\label{ques:Fano} Assume that $\bigwedge^{r} T_X$ is strictly nef for some $1\leq r\leq n$. Then is $X$ a Fano variety?
\end{question}
Finally, Theorem~\ref{them:str} follows from Theorem~\ref{them:CH19}, Proposition~\ref{prop:RC:Fano} and standard arguments.

\section{Preliminaries}

\subsection{Notation and Conventions}\label{subsec:NC} We will use the basic notation and definitions in \cite{Har}, \cite{Kb}, \cite{L1}, \cite{L2} and  \cite{KM}. Along this paper, we work over the complex number field.\begin{itemize}
\item A {\it curve} means a projective variety of dimension one. 
\item Let $X$ be a smooth projective variety. A line bundle $L$ on $X$ is said to be {\it strictly nef} (resp. {\it nef}) if the intersection number $L\cdot C$ is positive (resp. non-negative) for any curve $C\subset X$. In general, we say that a vector bundle $\cE$ is {\it strictly nef} (resp. {\it nef}) if the tautological line bundle $\cO_{\P(\cE)}(1)$ is strictly nef (resp. {\it nef}) on $\P(\cE)$. 
\item For a non-constant morphism $f: \P^1 \to X$ from a projective line $\P^1$ to a smooth projective variety $X$, $f$ is said to be {\it free} if $f^{\ast}T_X$ is nef.
\end{itemize}



Throughout this section, we always assume the following:

 \begin{assumption}\label{ass:r:nef} \rm Assume $X$ is a smooth projective variety of dimension $n$ such that the exterior power $\bigwedge^r T_X$ is nef for some $1\leq r<n$.\end{assumption}

\begin{proposition}\label{prop:fundamental:properties} The following hold:
\begin{enumerate} 
\item The anticanonical divisor $-K_X$ is nef.
\item If the Kodaira dimension $\kappa(X)=0$, then there exists a finite \'etale cover $f: \tilde{X} \to X$ such that $\tilde{X} $ is an Abelian variety.
\end{enumerate}
\end{proposition}

\begin{proof} The first part follows from $\det \left( \bigwedge^r T_X\right) \cong \cO_X\left( \binom{n-1}{r-1} \left( -K_X \right)\right)$. The second part follows from \cite[Theorem~1.1]{Yas14} (see also \cite[Proposition~1.2]{CP92}). 
\end{proof}

\begin{lemma}[{\cite[Lemma~1.3]{CP92}, \cite[Lemma~2.9]{Yas12}}]\label{lem:nonfree} Let $f: \P^1 \to X$ be a non-free rational curve, that is, $f^{\ast}T_X$ is not nef. Then we have $-K_X\cdot f_{\ast}(\P^1) \geq n-r+1$.
\end{lemma}

\begin{proof} Assume that the splitting type of $f^{\ast}T_X$ is $(a_1, a_2, \ldots, a_n)$, that is,  
$$
f^{\ast}T_X \cong \bigoplus_{i=1}^n \cO_{\P^1}(a_i) \,\,\,\,(a_1\geq a_2 \geq \ldots \geq a_n,~a_1\geq 2).
$$ 
The $r$-th exterior power 
$$\bigwedge^r f^{\ast}T_X \cong \bigoplus_{1\leq i_1<i_2<\ldots<i_r\leq n}^n \cO_{\P^1}(a_{i_1}+a_{i_2}+\ldots+a_{i_r})$$ is nef; this yields 
$$
a_{{n-r+1}}+a_{{n-r+2}}+\ldots+a_{n}\geq 0.
$$ Since $f$ is not free, $a_n$ is negative. These imply that 
$$
(r-1)a_{n-r+1}\geq a_{n-r+1}+a_{n-r+2}+\ldots +a_{n-1}\geq -a_n \geq 1.$$ 
Thus $a_{n-r+1}$ is positive. As a consequence, we have the inequality 
$$
-K_X\cdot f_{\ast}(\P^1)= a_1+(a_2+ \ldots + a_{n-r})+(a_{n-r+1}+\ldots + a_n) \geq 2+(n-r-1)+0=n-r+1.
$$ 
\end{proof}

\begin{proposition}[{\cite[Proposition~3.3]{Wat20}}]\label{prop:fiber:target} 
Let $\varphi: X \to A$ be a smooth morphism onto an Abelian variety with irreducible fibers. Then the following hold:
\begin{enumerate}
\item If $\dim A\geq r-1$, then $T_X$ is nef.
\item If $\dim A< r-1$, then $\bigwedge^{r-\dim A}T_{X/A}$ is nef.
\end{enumerate}
\end{proposition} 

\begin{proof} We have an exact sequence
\begin{align*}
0 \to T_{X/A} \to T_X \to \varphi^{\ast} T_A\to 0.  \tag{1}
\end{align*}
By \cite[Chapter~II, Exercise~5.16 (d)]{Har}, we have a filtration of $\bigwedge^r T_{X}$:
$$
\bigwedge^r T_{X}=E^0\supset E^1\supset E^2\supset  \ldots \supset E^{r+1}=0 
$$ such that 
$$
E^p/E^{p+1}\cong \left(\bigwedge^pT_{X/A}\right)\bigotimes \left(\bigwedge^{r-p} \varphi^{\ast}T_A\right)
$$ for any $p$. In particular, we have the following exact sequences:
\begin{align*}
0 \to E^1 \to \bigwedge^r T_{X} \to \bigwedge^{r} \varphi^{\ast}T_A\to 0 \tag{2}
\end{align*}
\begin{align*}
0 \to E^2 \to E^1 \to T_{X/A}\bigotimes \left(\bigwedge^{r-1}\varphi^{\ast}T_A\right)\to 0 \tag{3}
\end{align*}
To prove $\rm (i)$, assume $\dim A\geq r-1$. Remark that $T_A\cong \cO_A^{\oplus \dim A}$. We claim that $E^1$ is nef. If $\dim A\geq r$, then it follows from the sequence (2) and \cite[Proposition 1.2~(8)]{CP91} that $E^1$ is nef. If $\dim A=r-1$, then the sequence (2) yields $E^1\cong \bigwedge^r T_{X}$; this implies that $E^1$ is nef. By the sequence (3), $T_{X/A}\bigotimes \left(\bigwedge^{r-1}\varphi^{\ast}T_A\right)$ is nef. Since $\bigwedge^{r-1}\left(\varphi^{\ast}T_A\right)$ is trivial bundle, we conclude that the relative tangent bundle $T_{X/A}$ is nef. Finally our assertion follows from the sequence (1).

To prove $\rm (ii)$, assume $\dim A< r-1$. Since $\bigwedge^p \varphi^{\ast}T_A=0$ for any $p>\dim A$, we have 
$$
\bigwedge^r T_X=E^0=E^1=\cdots=E^{r-\dim A}.
$$ Thus we have a surjection $\bigwedge^r T_X=E^{r-\dim A}\to \bigwedge^{r-\dim A}T_{X/A}$; this implies that $\bigwedge^{r-\dim A}T_{X/A}$ is nef.
\end{proof}

\section{Proof of the Main Theorem}

The following is due to Gachet:

\begin{proposition}[{\cite[Theorem~1.2]{Gac22}}]\label{prop:n-1:Fano} 
Let $X$ be a smooth projective variety of dimension $n$. Assume that $X$ is rationally connected and $ \bigwedge^{n-1} T_X$ is nef. Then $-K_X$ is ample, that is, $X$ is a Fano variety. 
\end{proposition} 

\begin{proof} This follows from the same argument as in \cite[Lemma~3.1, Lemma~3.3]{Gac22}. Actually the proof works if we replace the assumption that $\bigwedge^{n-1} T_X$ is strictly nef by the assumption that $X$ is rationally connected and $\bigwedge^{n-1} T_X$ is nef; this yields that $-K_X$ is nef and big. Then we may conclude that $-K_X$ is ample by the same argument as in \cite[Lemma~3.3]{Gac22} and Lemma~\ref{lem:nonfree}.
\end{proof}

\begin{remark}\label{rem:Cecile} Although Proposition~\ref{prop:n-1:Fano} is not written explicitly in \cite{Gac22}, Gachet introduced this statement holds at Algebraic Geometry seminar of the University of Tokyo (see Acknowledgements below).
\end{remark}

\begin{theorem}[{\cite[Theorem~3.3]{LayNah05}, \cite{Fuj22}}]\label{them:Fuj} Let $X$ be a smooth projective variety of dimension $n$.  For a vector bundle $E$ of rank $r$, assume that its exterior power $\bigwedge^m E$ is nef for some positive integer $m$. Then the vector bundle $\bigwedge^{m+k} E$ is also nef for any $0\leq k\leq n-m$. 
\end{theorem}

\begin{remark}\label{rem:str:nef} In general, if a vector bundle $E$ is strictly nef, it is not necessarily that its exterior power $\bigwedge^rE$ is strictly nef. For instance, see \cite[Section~10~in~Chapter~I]{Hartshorne-ample} and \cite[Example~2.1]{LOY21}). This means that an analogue of Theorem~\ref{them:Fuj} does not hold if we replace nefness of $\bigwedge^m E$ by strictly nefness.
\end{remark}

\begin{proof}[{Proof of Proposition~\ref{prop:RC:Fano}}] Assume that $X$ is rationally connected and $\bigwedge^{r} T_X$ is nef for some $1\leq r<n$. Then Theorem~ \ref{them:Fuj} implies that $\bigwedge^{n-1} T_X$ is nef. Applying Proposition~\ref{prop:fiber:target}, we see that $X$ is a Fano variety.  
\end{proof}

\begin{proof}[{Proof of Theorem~\ref{them:str}}]  By Proposition~\ref{prop:fundamental:properties}~(i), $-K_X$ is nef; according to Theorem~\ref{them:CH19}, this turns out that there exists a finite \'etale cover $X'\to X$ such that $X'\cong Y\times Z$ where $K_Y$ is trivial and $Z$ is a locally trivial fibration $Z \to {\rm Alb}(Z)$ with a rationally connected fiber. Since we have $X'\to X$ is \'etale, $\bigwedge^{r} T_{X'}$ is also nef; then by Theorem~\ref{them:Fuj}, $\bigwedge^{n-1} T_{X'}$ is also nef. Let $p_1: X'\to Y$ (resp. $p_2: X'\to Z$) be the first projection (resp. the second projection). We denote by $\ell$ the dimension of $Y$. Since we have 
$$
\bigwedge^{n-1} T_{X'}\cong \left[p_1^{\ast}\left(\bigwedge^{\ell}T_Y \right)\bigotimes p_2^{\ast}\left(\bigwedge^{n-\ell-1}T_Z \right)\right]\bigoplus \left[p_1^{\ast}\left(\bigwedge^{\ell-1}T_Y \right)\bigotimes p_2^{\ast}\left(\bigwedge^{n-\ell}T_Z \right)\right],
$$
the direct summand $p_1^{\ast}\left(\bigwedge^{\ell-1}T_Y \right)\bigotimes p_2^{\ast}\left(\bigwedge^{n-\ell}T_Z \right)$ is nef; restricting this bundle to a fiber of the projection $p_2$, we see that $\bigwedge^{\ell-1}T_Y$ is also nef provided that $\ell >0$. 
If $\ell=1$, then $Y$ is an elliptic curve. Furthermore if $\ell>1$, then Proposition~\ref{prop:fundamental:properties}~(ii) implies that $Y$ is a finite \'etale quotient of an Abelian variety $\tilde{Y}$. Hence, in any case, there exists a finite \'etale cover $\tilde{X}\to X'$ such that $\tilde{X}$ is a locally trivial fibration $\varphi: \tilde{X} \to A$ onto an Abelian variety $A$ with a rationally connected fiber. Then our assertion follows from Proposition~\ref{prop:fiber:target} and Proposition~\ref{prop:RC:Fano}.
\end{proof}

\section*{Acknowledgements} On August the 1st 2022, C\'ecile Gachet gave a talk on the result of \cite{Gac22} in Algebraic Geometry seminar of the University of Tokyo; the author knew her excellent result \cite[Theorem~1.2]{Gac22} at the time. The author would like to thank C\'ecile Gachet for sending him her preprint \cite{Gac22} and answering his questions. The author would like to thank Kento Fujita for sending him his private note \cite{Fuj22}, which states that for a vector bundle $E$ of rank $r$ if $\bigwedge^{\ell} E$ is nef, then so is $\bigwedge^m E$ for any $\ell\leq m\leq r$. On the other hand, while preparing this paper, the author noticed the result of \cite{Fuj22} was contained in \cite{LayNah05}.
 
 
\bibliographystyle{plain}
\bibliography{biblio}

\begin{thebibliography}{10}

\bibitem{CP91}
Fr{\'e}d{\'e}ric Campana and Thomas Peternell.
\newblock Projective manifolds whose tangent bundles are numerically effective.
\newblock {\em Math. Ann.}, 289(1):169--187, 1991.

\bibitem{CP92}
Fr\'{e}d\'{e}ric Campana and Thomas Peternell.
\newblock On the second exterior power of tangent bundles of threefolds.
\newblock {\em Compositio Math.}, 83(3):329--346, 1992.

\bibitem{CH19}
Junyan Cao and Andreas H\"{o}ring.
\newblock A decomposition theorem for projective manifolds with nef
  anticanonical bundle.
\newblock {\em J. Algebraic Geom.}, 28(3):567--597, 2019.

\bibitem{Gac22}
Gachet C\'ecile.
\newblock Positivity of higher exterior powers of the tangent bundle.
\newblock Preprint arXiv:{\tt 2207.10854}, 2022.

\bibitem{DPS94}
Jean-Pierre Demailly, Thomas Peternell, and Michael Schneider.
\newblock Compact complex manifolds with numerically effective tangent bundles.
\newblock {\em J. Algebraic Geom.}, 3(2):295--345, 1994.

\bibitem{Fuj22}
Kento Fujita.
\newblock Remarks on ampleness of wedge vector bundles.
\newblock private note.

\bibitem{Hartshorne-ample}
Robin Hartshorne.
\newblock {\em Ample subvarieties of algebraic varieties}.
\newblock Lecture Notes in Mathematics, Vol. 156. Springer-Verlag, Berlin,
  1970.
\newblock Notes written in collaboration with C. Musili.

\bibitem{Har}
Robin Hartshorne.
\newblock {\em Algebraic geometry}.
\newblock Springer-Verlag, New York-Heidelberg, 1977.
\newblock Graduate Texts in Mathematics, No. 52.

\bibitem{Kb}
J{\'a}nos Koll{\'a}r.
\newblock {\em Rational curves on algebraic varieties}, volume~32 of {\em
  Ergebnisse der Mathematik und ihrer Grenzgebiete. 3. Folge. A Series of
  Modern Surveys in Mathematics [Results in Mathematics and Related Areas. 3rd
  Series. A Series of Modern Surveys in Mathematics]}.
\newblock Springer-Verlag, Berlin, 1996.

\bibitem{KM}
J{\'a}nos Koll{\'a}r and Shigefumi Mori.
\newblock {\em Birational geometry of algebraic varieties}, volume 134 of {\em
  Cambridge Tracts in Mathematics}.
\newblock Cambridge University Press, Cambridge, 1998.
\newblock With the collaboration of C. H. Clemens and A. Corti, Translated from
  the 1998 Japanese original.

\bibitem{LayNah05}
F.~Laytimi and W.~Nahm.
\newblock A vanishing theorem.
\newblock {\em Nagoya Math. J.}, 180:35--43, 2005.

\bibitem{L1}
Robert Lazarsfeld.
\newblock {\em Positivity in algebraic geometry. {I}}, volume~48 of {\em
  Ergebnisse der Mathematik und ihrer Grenzgebiete. 3. Folge. A Series of
  Modern Surveys in Mathematics [Results in Mathematics and Related Areas. 3rd
  Series. A Series of Modern Surveys in Mathematics]}.
\newblock Springer-Verlag, Berlin, 2004.
\newblock Classical setting: line bundles and linear series.

\bibitem{L2}
Robert Lazarsfeld.
\newblock {\em Positivity in algebraic geometry. {II}}, volume~49 of {\em
  Ergebnisse der Mathematik und ihrer Grenzgebiete. 3. Folge. A Series of
  Modern Surveys in Mathematics [Results in Mathematics and Related Areas. 3rd
  Series. A Series of Modern Surveys in Mathematics]}.
\newblock Springer-Verlag, Berlin, 2004.
\newblock Positivity for vector bundles, and multiplier ideals.

\bibitem{LOY19}
Duo Li, Wenhao Ou, and Xiaokui Yang.
\newblock On projective varieties with strictly nef tangent bundles.
\newblock {\em J. Math. Pures Appl. (9)}, 128:140--151, 2019.

\bibitem{LOY21}
Jie Liu, Wenhao Ou, and Xiaokui Yang.
\newblock Strictly nef vector bundles and characterizations of {$\Bbb{P}^n$}.
\newblock {\em Complex Manifolds}, 8(1):148--159, 2021.

\bibitem{Mori79}
Shigefumi Mori.
\newblock Projective manifolds with ample tangent bundles.
\newblock {\em Ann. of Math. (2)}, 110(3):593--606, 1979.

\bibitem{SiuYau80}
Yum~Tong Siu and Shing~Tung Yau.
\newblock Compact {K}\"ahler manifolds of positive bisectional curvature.
\newblock {\em Invent. Math.}, 59(2):189--204, 1980.

\bibitem{Wat20}
Kiwamu Watanabe.
\newblock Fano manifolds of coindex three admitting nef tangent bundle.
\newblock {\em Geom. Dedicata}, 210:165--178, 2021.

\bibitem{Wat21}
Kiwamu Watanabe.
\newblock Positivity of the second exterior power of the tangent bundles.
\newblock {\em Adv. Math.}, 385:Paper No. 107757, 27, 2021.

\bibitem{Yas12}
Kazunori Yasutake.
\newblock On the second exterior power of tangent bundles of {F}ano fourfolds
  with picard number $\rho_x\geq 2$.
\newblock Preprint arXiv:{\tt 1212.0685}, 2012.

\bibitem{Yas14}
Kazunori Yasutake.
\newblock On the second and third exterior power of tangent bundles of {F}ano
  manifolds with birational contractions.
\newblock Preprint arXiv:{\tt 1403.5304}, 2014.

\end{thebibliography}
\end{document}